\documentclass[11pt,reqno]{amsart}

\usepackage{amsmath,amssymb,amsthm}
\usepackage[margin=1.15in]{geometry}
\usepackage[colorlinks=true,linkcolor=blue,citecolor=blue,urlcolor=blue]{hyperref}

\theoremstyle{plain}
\newtheorem{theorem}{Theorem}[section]
\newtheorem{corollary}[theorem]{Corollary}
\newtheorem{lemma}[theorem]{Lemma}
\newtheorem{proposition}[theorem]{Proposition}

\newtheorem{maintheorem}{Theorem}

\newtheorem{mainconj}[maintheorem]{Conjecture}

\theoremstyle{definition}
\newtheorem{definition}[theorem]{Definition}
\newtheorem{example}[theorem]{Example}

\theoremstyle{remark}
\newtheorem{remark}[theorem]{Remark}

\numberwithin{equation}{section}

\newcommand{\R}{\mathbb{R}}
\newcommand{\Hyp}{\mathbb{H}}
\newcommand{\Gt}{\widetilde{\Gamma}}
\newcommand{\Mt}{\widetilde{M}}
\newcommand{\htil}{\widetilde{h}}
\newcommand{\cV}{\mathcal{V}}
\DeclareMathOperator{\length}{length}
\DeclareMathOperator{\area}{Area}
\DeclareMathOperator{\vol}{Vol}
\DeclareMathOperator{\Ent}{Ent}
\DeclareMathOperator{\sys}{sys}
\DeclareMathOperator{\Br}{Br}

\title[Ball growth at all scales in universal covers of graphs]
{Ball growth at all scales in universal covers \\ of short metric graphs}

\author{Steve Karam}
\address{College of Engineering and Technology, American University of the Middle East, Egaila 54200, Kuwait}
\email{steve.karam@aum.edu.kw}

\subjclass[2020]{Primary 53C23; Secondary 05C12, 20F65, 37B40}
\keywords{Metric graphs, universal cover, growth of balls, volume entropy, trivalent tree,
systolic geometry, Guth's question}

\begin{document}

\begin{abstract}
Let $\Gamma$ be a connected metric graph of first Betti number $b\ge 2$, and let
$\Gamma_b$ be a trivalent graph of the same Betti number with all edges of unit length.
We show that if $\length(\Gamma)\le\lambda\,\length(\Gamma_b)$, then the universal cover
of $\Gamma$ contains a vertex around which, \emph{at every radius} $R\ge 0$, the ball is at
least as long as the ball of radius $R$ in the trivalent tree whose edges all have length
$3\lambda$. In particular, for $\lambda\le\frac13$, this ball is at least as long as the
ball of the same radius in the universal cover of $\Gamma_b$, with no multiplicative loss.
This improves an earlier theorem of the author, which gives the weaker lower bound
$(1-3\lambda)$ times the ball in the universal cover of $\Gamma_b$, and only for
$\lambda\le\frac16$. The proof rests on an elementary comparison principle. Among trees whose
vertices have degree at least three and whose edges have length at most $c$, the trivalent
tree with all edges of length $c$ has the smallest balls at every radius.
As an application, we double the exponential growth rate in the author's earlier uniform
lower bound for the area of balls in the universal cover of small-area surfaces of hyperbolic
type, from $\ln 2$ to $2\ln 2$, beyond the hyperbolic rate. We also obtain explicit constants
in the positive answer to Guth's question for surfaces.
Finally, we conjecture the sharp, all-scale version of the minimal volume entropy inequality
of Kapovich--Nagnibeda and Lim, and show that it would be optimal.
\end{abstract}

\maketitle

\noindent\textbf{Note added in version 2 (November 2026).}
This preprint is superseded by \href{https://arxiv.org/abs/2611.03277}{arXiv:2611.03277}, joint with S.~Sabourau.
It will not be submitted for publication.

\section{Introduction}\label{sec:intro}

\subsection{Background}\label{subsec:background}
Let $(M,h)$ be a closed Riemannian manifold, let $(\Mt,\htil)$ be its universal cover with
the lifted metric, and denote by $B_{\htil}(\tilde x,R)$ the closed ball of radius $R$ centred at
$\tilde x\in\Mt$. The \emph{volume entropy}
\[
\Ent(M,h)=\lim_{R\to\infty}\frac{1}{R}\log\vol B_{\htil}(\tilde x,R)
\]
exists and does not depend on $\tilde x$ \cite{Man79}. It measures the \emph{asymptotic}
exponential growth of balls in $\Mt$. Suppose that $M$ carries a hyperbolic metric $hyp$, that is,
a Riemannian metric of constant curvature $-1$, so that $(\Mt,\widetilde{hyp})$ is the hyperbolic
space $\Hyp^n$. The theorem of Besson, Courtois and Gallot \cite{BCG95} (Katok \cite{Kat82} in
dimension two) shows that every metric $h$ with $\vol(M,h)=\vol(M,hyp)$ satisfies
$\Ent(M,h)\ge\Ent(M,hyp)$. This is a statement about balls of very large radius only.

Guth \cite{Gut11} proved a complementary statement at a fixed scale: there is a constant
$\delta(n)>0$ such that if $\vol(M,h)<\delta(n)\vol(M,hyp)$, then $\Mt$ contains a ball of radius
$1$ whose volume is larger than that of a ball of radius $1$ in $\Hyp^n$. He asked whether a
single constant $\delta(n)$ works for \emph{all} radii $R\ge1$ simultaneously. In \cite{Kar15}, the
author answered this question positively for surfaces. Sabourau \cite{Sab22} answered it in all
dimensions, for closed manifolds admitting a metric of negative curvature, and Alpert
\cite{Alp24} recently relaxed the volume hypothesis.

The approach of \cite{Kar15} to surfaces goes through metric graphs, where the analogous
questions have independent interest. We now fix the notation for graphs.

\subsection{Metric graphs and notation}\label{subsec:notation}
A \emph{graph} is a one-dimensional CW-complex, possibly with loops and multiple edges; unless
stated otherwise it is finite. The \emph{degree} $\deg v$ of a vertex $v$ is the number of edges
incident to it, loops being counted twice. A graph is \emph{trivalent} if all its vertices have
degree~$3$. A \emph{subgraph} is a union of vertices and edges. The \emph{first Betti number} of a
finite connected graph $\Gamma$ with $e$ edges and $v$ vertices is $b(\Gamma)=e-v+1$; it is the rank
of $H_1(\Gamma;\mathbb Z)$.

A \emph{metric graph} $(\Gamma,h)$ is a graph in which each edge is identified with a compact
interval of positive length, endowed with the induced length metric $d$. The length of a subset
$A\subset\Gamma$, denoted $\length(A)$, is its one-dimensional Hausdorff measure; in particular
$\length(\Gamma,h)$ is the sum of the edge lengths. A \emph{metric tree} is a connected, simply
connected, locally finite metric graph, possibly infinite. The universal cover of a connected
metric graph $(\Gamma,h)$, with the lifted metric, is a metric tree denoted $(\Gt,\htil)$; its
vertices are the lifts of the vertices of $\Gamma$. For a metric graph or tree $X$, a point $x\in X$
and $R\ge0$, we write $B_X(x,R)=\{y\in X:d(x,y)\le R\}$, and $B_{\htil}(\tilde x,R):=B_{\Gt}(\tilde x,R)$.
As for manifolds, the \emph{volume entropy} of a connected metric graph is
\[
\Ent(\Gamma,h)=\lim_{R\to\infty}\frac{1}{R}\log\length B_{\htil}(\tilde x,R),
\]
which exists and does not depend on $\tilde x$ \cite{Man79}.

For every $b\ge2$ we fix a connected trivalent graph $\Gamma_b$ of first Betti number $b$, with all
edges of length $1$. A trivalent graph satisfies $2e=3v$, so $\Gamma_b$ has $3b-3$ edges and
$\length(\Gamma_b)=3b-3$. For $c>0$, let $T_c$ be the infinite trivalent metric tree all of whose edges
have length $c$; it is unique up to isometry, and the universal cover of $\Gamma_b$ is $T_1$ for every
choice of $\Gamma_b$.

\begin{definition}[Model function]\label{def:Vc}
For $c>0$ and $R\ge0$, let $\cV_c(R):=\length B_{T_c}(v,R)$, where $v$ is a vertex of $T_c$. Since
$T_c$ is vertex-transitive, $\cV_c(R)$ does not depend on $v$. For $j\ge0$, exactly $3\cdot2^j$ points of
$T_c$ lie at distance $r\in(jc,(j+1)c)$ from $v$, so, with $k=\lfloor R/c\rfloor$,
\[
\cV_c(R)=3c\,(2^k-1)+3\,(R-kc)\,2^k,\qquad\text{and}\qquad \cV_c(R)=c\,\cV_1(R/c).
\]
\end{definition}

The function $\cV_1$ is the function $V'_{(\Gt_b,\htil_b)}$ of \cite{Kar15}. Kapovich--Nagnibeda
\cite{KN07} and Lim \cite{Lim08} (see also \cite{Riv10,McM15} and the elementary proof in
\cite{Sab26}) proved that for every connected metric graph $(\Gamma,h)$ with $b(\Gamma)=b\ge2$,
\begin{equation}\label{eq:KNL}
\Ent(\Gamma,h)\cdot\length(\Gamma,h)\ \ge\ (3b-3)\ln 2=\Ent(\Gamma_b)\cdot\length(\Gamma_b).
\end{equation}
Thus $\Gamma_b$ plays for graphs the role that the hyperbolic metric plays for manifolds. In
\cite{Kar15}, the author proved the following all-scale graph analogue of Guth's question.

\begin{theorem}[{\cite[Theorem C]{Kar15}}]\label{thm:Karam}
Let $\lambda\in(0,\frac16]$ and let $(\Gamma,h)$ be a connected metric graph with $b(\Gamma)=b\ge2$ and
$\length(\Gamma,h)\le\lambda\length(\Gamma_b)$. Then there is a vertex $\tilde u$ of $\Gt$ such that
$\length B_{\htil}(\tilde u,R)\ge(1-3\lambda)\,\cV_1(R)$ for every $R\ge0$.
\end{theorem}

The constant $(1-3\lambda)$ comes from a bi-Lipschitz embedding of a rescaled copy of $T_1$ in
$\Gt$. Such an embedding loses a multiplicative constant and, more importantly, does not
see that a shorter graph should have \emph{faster} growth.

\subsection{Main results}\label{subsec:results}
Our main theorem removes the multiplicative loss, extends the range of $\lambda$, and
gives the growth rate one expects from scaling.

\begin{maintheorem}\label{thmA}
Let $(\Gamma,h)$ be a connected metric graph with $b(\Gamma)=b\ge2$, and let $\lambda>0$ be such that
\[
\length(\Gamma,h)\ \le\ \lambda\,\length(\Gamma_b)=3\lambda(b-1).
\]
Then there exists a vertex $\tilde u$ of $\Gt$ such that, for every $R\ge0$,
\[
\length B_{\htil}(\tilde u,R)\ \ge\ \cV_{3\lambda}(R)=3\lambda\,\cV_1\!\Big(\frac{R}{3\lambda}\Big)\ \ge\ 9\lambda\big(2^{R/(3\lambda)}-1\big).
\]
In particular, if $\lambda\le\frac13$, then $\length B_{\htil}(\tilde u,R)\ge\cV_1(R)$ for every $R\ge0$.
\end{maintheorem}

For $\lambda\le\frac13$ we have $\cV_{3\lambda}\ge\cV_1\ge(1-3\lambda)\cV_1$ (Corollary~\ref{cor:monotone}),
so Theorem~\ref{thmA} contains Theorem~\ref{thm:Karam}. It improves on it in three ways:
\begin{enumerate}
\item the range is $\lambda\le\frac13$ instead of $\lambda\le\frac16$;
\item there is no multiplicative loss;
\item the exponential growth rate is $\frac{\ln2}{3\lambda}$ instead of $\ln 2$.
\end{enumerate}
Moreover, the proof shows that $\tilde u$ can be taken to be any lift of any vertex of degree at
least $3$ of an explicit subgraph of $\Gamma$ (Section~\ref{sec:proofA}).

The key new ingredient is the following comparison principle, which is sharp and may be of
independent interest.

\begin{maintheorem}\label{thmB}
Let $c>0$ and let $T$ be a metric tree in which every vertex has degree at least $3$ and every
edge has length at most $c$. Assume that the edge lengths of $T$ are bounded below by a positive
constant. Then for every vertex $v$ of $T$ and every $R\ge0$,
\[
\length B_T(v,R)\ \ge\ \cV_c(R).
\]
Equality holds for all $R$ when $T=T_c$.
\end{maintheorem}

In particular, $c\mapsto\cV_c(R)$ is non-increasing (Corollary~\ref{cor:monotone}). In the setting of
\cite[Proposition~3.1]{Kar15}, Theorem~\ref{thmB} shows that \emph{every} vertex of the tree has
large balls, with the optimal constant (Remark~\ref{rem:Karam31}).

\medskip
\noindent\textbf{Application to surfaces.}
For a closed surface $M$ with a Riemannian metric $h$, put
\[
V_{(\Mt,\htil)}(R):=\sup_{\tilde x\in\Mt}\area B_{\htil}(\tilde x,R),
\]
and recall that $V_{\Hyp^2}(R):=\area B_{\Hyp^2}(x,R)=2\pi(\cosh R-1)$ for any $x\in\Hyp^2$.
In \cite{Kar15}, the lower bounds on $V_{(\Mt,\htil)}$ are deduced from Theorem~\ref{thm:Karam},
applied with $\lambda=\frac16$ to a short graph in $M$ that carries its homology. Using
Theorem~\ref{thmA} instead doubles the exponential rate, which then exceeds the hyperbolic rate~$1$.

\begin{maintheorem}\label{thmC}
Let $(M,hyp)$ be a closed hyperbolic surface and $h$ a Riemannian metric on $M$.
\begin{enumerate}
\item If $\area(M,h)\le\frac{1}{2^{13}\pi}\area(M,hyp)$, then for every $R\ge0$,
\[
V_{(\Mt,\htil)}(R)\ \ge\ \frac32\Big(\frac{4^R-1}{\ln4}-R\Big);
\]
in particular $V_{(\Mt,\htil)}(R)\ge V_{\Hyp^2}(R)$ for every $R\ge3$.
\item If $\area(M,h)\le\frac{1}{2^{17}\pi}\area(M,hyp)$, then
$V_{(\Mt,\htil)}(R)\ge V_{\Hyp^2}(R)$ for every $R\ge1$.
\end{enumerate}
\end{maintheorem}

Under the hypothesis of part~(1), the lower bound of \cite[Theorem~9.2]{Kar15} is
$\frac{1}{2\ln2}(\cosh(R\ln2)-1)$, which grows like $2^R$. It therefore never dominates
$V_{\Hyp^2}(R)$ without rescaling. The constant $\delta$ answering Guth's question in \cite{Kar15}
is not made explicit; part~(2) provides an explicit one.

\medskip
\noindent\textbf{A sharp conjecture.}
Theorem~\ref{thmA} is sharp up to a factor of $3$ in the scale, and we believe this factor is
an artefact of the method (Section~\ref{sec:sharp}).

\begin{mainconj}\label{conjD}
Let $(\Gamma,h)$ be a connected metric graph with $b(\Gamma)=b\ge2$, and let $\lambda>0$ be such that
$\length(\Gamma,h)\le\lambda\,\length(\Gamma_b)$. Then there is a vertex $\tilde u$ of $\Gt$ such that
\[
\length B_{\htil}(\tilde u,R)\ \ge\ \cV_{\lambda}(R)\qquad\text{for every }R\ge0.
\]
\end{mainconj}

Conjecture~\ref{conjD} would be a scale-by-scale strengthening of \eqref{eq:KNL}
(Remark~\ref{rem:entropy}), and it would be optimal: $\Gamma_b$ with all edges of length $\lambda$
attains equality at every radius (Proposition~\ref{prop:optimal}). Theorem~\ref{thmA} proves it with
$\lambda$ replaced by $3\lambda$. A \emph{theta graph} consists of two vertices joined by three edges.
We show that the conjecture is attained with equality, at a positive radius, by a two-parameter
family of theta graphs that are not equilateral (Example~\ref{ex:theta}).

\subsection{Organisation}
Section~\ref{sec:prelim} collects preliminaries on arcs, subtrees and the function $\cV_c$.
Section~\ref{sec:comparison} proves Theorem~\ref{thmB}. Section~\ref{sec:core} extracts a short
core from a short graph. Section~\ref{sec:proofA} proves Theorem~\ref{thmA}.
Section~\ref{sec:surfaces} proves Theorem~\ref{thmC}, and Section~\ref{sec:sharp} discusses
sharpness and Conjecture~\ref{conjD}.

\section{Preliminaries}\label{sec:prelim}

\subsection{Arcs and smoothing}

\begin{definition}\label{def:arc}
Let $\Gamma$ be a finite connected graph with $b(\Gamma)\ge2$ and no vertex of degree one. Then $\Gamma$ has
vertices of degree at least $3$, which we call \emph{branch vertices}. Indeed, if all degrees were
$2$, then $\Gamma$ would be a cycle and $b(\Gamma)=1$. An \emph{arc} of $\Gamma$ is a path whose endpoints are
branch vertices (possibly equal) and whose interior vertices all have degree $2$. Every edge lies in
exactly one arc. \emph{Smoothing} $\Gamma$ means forgetting the vertices of degree $2$, so that each arc
becomes a single edge of the same length. This yields a graph structure in which every vertex has
degree at least $3$, on the same metric space.
\end{definition}

\subsection{Subtrees of universal covers}

\begin{lemma}\label{lem:subtree}
Let $\Gamma$ be a finite connected metric graph, $\Gamma_0\subset\Gamma$ a connected subgraph, and $T$ a
connected component of the preimage of $\Gamma_0$ in $\Gt$. Then $T$ is isometric to the universal cover
of $\Gamma_0$, and the intrinsic metric of $T$ is the restriction of the metric of $\Gt$. In particular,
$B_T(x,R)=B_{\Gt}(x,R)\cap T$ for all $x\in T$ and $R\ge0$, and
\[
\length B_{\Gt}(x,R)\ \ge\ \length B_T(x,R).
\]
\end{lemma}

\begin{proof}
The inclusion $\Gamma_0\hookrightarrow\Gamma$ induces an injection of fundamental groups, so $T$ is a simply
connected covering space of $\Gamma_0$, that is, its universal cover. A connected subset of a tree contains
the unique geodesic between any two of its points, so the intrinsic metric of $T$ coincides with the
restriction of the metric of $\Gt$.
\end{proof}

\subsection{Branches and the model function}
Let $T$ be a metric tree. An \emph{oriented edge} $e=(v,w)$ of $T$ is an edge together with a choice
of initial vertex $v$ and terminal vertex $w$; we write $\ell(e)$ for its length. The \emph{branch}
of $T$ defined by $e$ is the closure $T_e$ of the connected component of $T\setminus\{v\}$ that
contains the interior of $e$.

For $c>0$, define $\beta_c\colon[0,\infty)\to[0,\infty)$ recursively by
\begin{equation}\label{eq:beta}
\beta_c(R)=R\quad(0\le R\le c),\qquad\beta_c(R)=c+2\beta_c(R-c)\quad(R>c).
\end{equation}
For every oriented edge $e=(v,w)$ of $T_c$, the branch $T_e$ satisfies
$\length\{x\in T_e:d(v,x)\le R\}=\beta_c(R)$. A vertex of $T_c$ has three branches, which meet only at
the vertex, so
\begin{equation}\label{eq:V=3beta}
\cV_c=3\beta_c.
\end{equation}

\begin{lemma}\label{lem:beta}
For all $c>0$ and $R,R',s\ge0$:
\begin{enumerate}
\item $\beta_c(R)=c\,\beta_1(R/c)$ and $\beta_1(s)=2^{\lfloor s\rfloor}-1+(s-\lfloor s\rfloor)\,2^{\lfloor s\rfloor}$;
\item $\beta_c$ is continuous, convex and piecewise linear, with slope $2^k$ on $(kc,(k+1)c)$; in particular,
$\beta_c(R)-\beta_c(R')\ge R-R'$ whenever $R\ge R'$;
\item $\beta_1(s)\ge2^s-1$; consequently $\cV_c(R)\ge3c\,(2^{R/c}-1)$.
\end{enumerate}
\end{lemma}

\begin{proof}
Parts (1) and (2) follow from \eqref{eq:beta} by induction on $\lfloor R/c\rfloor$. For (3), $\beta_1$ agrees
with the convex function $s\mapsto2^s-1$ at the integers and is affine in between, so it lies above it. The
consequence follows from (1) and \eqref{eq:V=3beta}.
\end{proof}

\section{The comparison principle}\label{sec:comparison}

\begin{proof}[Proof of Theorem~\ref{thmB}]
Let $m>0$ be a lower bound for the edge lengths of $T$. For an oriented edge $e=(v,w)$ of $T$, put
\[
\Br_e(R):=\length\{x\in T_e : d(v,x)\le R\}.
\]
We prove that $\Br_e(R)\ge\beta_c(R)$ for every oriented edge $e$ and every $R\ge0$, by induction
on $k=\lfloor R/m\rfloor$. Let $\ell=\ell(e)\in[m,c]$.

\emph{Case $R\le\ell$.} Then $\Br_e(R)=R=\beta_c(R)$, because $R\le\ell\le c$. This includes $k=0$.

\emph{Case $R>\ell$.} Since $\deg w\ge3$, there are at least two oriented edges $e'=(w,w')$ with
$e'\neq(w,v)$, and their branches $T_{e'}$ are contained in $T_e$ and meet only at $w$. Since
$R-\ell\le R-m$, the induction hypothesis applies to them at radius $R-\ell$, and
\[
\Br_e(R)\ \ge\ \ell+2\beta_c(R-\ell)=:\varphi(\ell).
\]
If $R\le c$, then $R-\ell\le c$ and $\varphi(\ell)=\ell+2(R-\ell)\ge R=\beta_c(R)$.
If $R>c$, then $\beta_c(R)=\varphi(c)$ by \eqref{eq:beta}. Since $R-\ell\ge R-c$,
Lemma~\ref{lem:beta}(2) gives
\[
\varphi(\ell)-\varphi(c)=(\ell-c)+2\big(\beta_c(R-\ell)-\beta_c(R-c)\big)\ \ge\ (\ell-c)+2(c-\ell)=c-\ell\ \ge\ 0.
\]
This completes the induction. Finally, $B_T(v,R)$ is the union of the sets
$\{x\in T_e:d(v,x)\le R\}$ over the $\deg v\ge3$ oriented edges $e$ with initial vertex $v$, which meet
only at $v$. Hence $\length B_T(v,R)\ge3\beta_c(R)=\cV_c(R)$ by \eqref{eq:V=3beta}. Equality for $T=T_c$
holds by Definition~\ref{def:Vc}.
\end{proof}

\begin{corollary}\label{cor:monotone}
For $0<c\le c'$ and $R\ge0$, we have $\cV_c(R)\ge\cV_{c'}(R)$. In particular, $\cV_c(R)\ge\cV_1(R)$ for
$c\le1$.
\end{corollary}

\begin{proof}
Apply Theorem~\ref{thmB} to $T=T_c$, whose edges have length $c\le c'$, with the edge bound $c'$.
\end{proof}

\begin{remark}\label{rem:Karam31}
Let $\Gamma$ be a finite connected metric graph with $b(\Gamma)\ge2$, all vertices of degree at least $3$,
and all edges of length at most $c$. For every $C'\ge c$, \cite[Proposition~3.1]{Kar15} produces a
vertex $\tilde u$ of $\Gt$ with $\length B_{\htil}(\tilde u,(C'+c)R)\ge C'\cV_1(R)$ for all $R\ge0$. By
Theorem~\ref{thmB} and Corollary~\ref{cor:monotone}, \emph{every} vertex $\tilde u$ of $\Gt$ satisfies
\[
\length B_{\htil}(\tilde u,R)\ \ge\ \cV_c(R)\ \ge\ \cV_{C'+c}(R)=(C'+c)\,\cV_1\!\Big(\frac{R}{C'+c}\Big)\ \ge\ C'\,\cV_1\!\Big(\frac{R}{C'+c}\Big).
\]
\end{remark}

\section{Short cores}\label{sec:core}

The following lemma isolates the combinatorial content of the induction in
\cite[Section~4]{Kar15}.

\begin{lemma}\label{lem:core}
Let $(\Gamma,h)$ be a finite connected metric graph with $b(\Gamma)\ge2$, and let $c>0$ be such that
$\length(\Gamma,h)\le c\,(b(\Gamma)-1)$. Then $\Gamma$ contains a connected subgraph $\Gamma_0$ such that
\begin{enumerate}
\item $b(\Gamma_0)\ge2$ and $\length(\Gamma_0)\le c\,(b(\Gamma_0)-1)$;
\item $\Gamma_0$ has no vertex of degree one;
\item every arc of $\Gamma_0$ (Definition~\ref{def:arc}) has length at most $c$.
\end{enumerate}
\end{lemma}

\begin{proof}
We argue by induction on $b=b(\Gamma)$. First, repeatedly delete vertices of degree one together
with their incident edges. This changes neither $b$ nor connectedness and does not increase the
length, so we may assume that $\Gamma$ has no vertex of degree one. If every arc of $\Gamma$ has
length at most $c$, take $\Gamma_0=\Gamma$.

If $b=2$, this is always the case. Indeed, $\Gamma$ then has at least two arcs, each of positive
length, so every arc is strictly shorter than $\length(\Gamma)\le c$. This settles the base case.

Let $b\ge3$ and suppose some arc $a$ has $\length(a)>c$. Remove the edges and the interior vertices
of $a$. The remaining graph has length $\length(\Gamma)-\length(a)<c(b-1)-c=c(b-2)$.

\emph{If $a$ does not disconnect $\Gamma$}, the remainder $\Gamma'$ is connected with
$b(\Gamma')=b-1\ge2$ and $\length(\Gamma')<c\,(b(\Gamma')-1)$. By induction, $\Gamma'$, and hence $\Gamma$,
contains the required subgraph.

\emph{If $a$ disconnects $\Gamma$}, the remainder is a disjoint union $\Gamma'\sqcup\Gamma''$ of two connected
graphs with $b(\Gamma')+b(\Gamma'')=b$. In each of $\Gamma',\Gamma''$, every vertex has degree at least $2$: the
endpoint of $a$ it contains had degree at least $3$, and all other degrees are unchanged. Hence both
contain a cycle, so $b(\Gamma'),b(\Gamma'')\ge1$, and both have positive length. Since
\[
\length(\Gamma')+\length(\Gamma'')\ <\ c\,(b(\Gamma')-1)+c\,(b(\Gamma'')-1),
\]
one of them, say $\Gamma'$, satisfies $\length(\Gamma')<c\,(b(\Gamma')-1)$. As its length is positive,
$b(\Gamma')\ge2$, and the induction hypothesis applied to $\Gamma'$ concludes the proof.
\end{proof}

\begin{remark}\label{rem:core-sharp}
The constant $c$ in Lemma~\ref{lem:core}(3) cannot be lowered in general. Let $\Gamma$ be the
\emph{rose} with $b$ petals, that is, one vertex with $b$ loops, each of length $c(b-1)/b$. Then
$\length(\Gamma)=c(b-1)$, and every subgraph satisfying (1) and (2) consists of the vertex and at least two
of the loops, whose arcs have length $c(b-1)/b$, which tends to $c$ as $b\to\infty$.
\end{remark}

\section{Proof of Theorem~\ref{thmA}}\label{sec:proofA}

Put $c:=3\lambda$, so that the hypothesis reads $\length(\Gamma,h)\le c\,(b-1)$. Let $\Gamma_0\subset\Gamma$ be
the subgraph given by Lemma~\ref{lem:core}, and let $T\subset\Gt$ be a connected component of the preimage
of $\Gamma_0$. By Lemma~\ref{lem:subtree}, $T$ is isometric to the universal cover of $\Gamma_0$, and its intrinsic
metric is the restriction of the metric of $\Gt$.

Smooth $\Gamma_0$ (Definition~\ref{def:arc}) and give $T$ the lifted graph structure. This changes neither the
metric space $\Gamma_0$ nor $T$. The vertices of $T$ are now the lifts of the branch vertices of $\Gamma_0$. They have
degree at least $3$ and are vertices of $\Gt$. The edges of $T$ are the lifts of the arcs of $\Gamma_0$, so they
have length at most $c$; their lengths take finitely many positive values and are therefore bounded below.
Theorem~\ref{thmB} and Lemma~\ref{lem:subtree} give, for \emph{every} lift $\tilde u$ of a branch vertex of
$\Gamma_0$ and every $R\ge0$,
\[
\length B_{\htil}(\tilde u,R)\ \ge\ \length B_T(\tilde u,R)\ \ge\ \cV_c(R)=\cV_{3\lambda}(R).
\]
The identity $\cV_{3\lambda}(R)=3\lambda\,\cV_1(R/3\lambda)$ is part of Definition~\ref{def:Vc}, and the lower
bound $9\lambda(2^{R/(3\lambda)}-1)$ is Lemma~\ref{lem:beta}(3). When $\lambda\le\frac13$,
Corollary~\ref{cor:monotone} gives $\cV_{3\lambda}\ge\cV_1$. \qed

\begin{remark}[Volume entropy]\label{rem:entropy}
Theorem~\ref{thmA} implies $\Ent(\Gamma,h)\ge\frac{\ln2}{3\lambda}$ with $\lambda=\frac{\length(\Gamma,h)}{3b-3}$, that
is, $\Ent(\Gamma,h)\cdot\length(\Gamma,h)\ge(b-1)\ln2$. This is weaker than \eqref{eq:KNL} by exactly the factor
$3$ discussed in Section~\ref{sec:sharp}. In the same way, Conjecture~\ref{conjD} would imply \eqref{eq:KNL}.
\end{remark}

\section{Application to surfaces: proof of Theorem~\ref{thmC}}\label{sec:surfaces}

We use the following notions from \cite{Kar15}. The \emph{systole} $\sys(M,h)$ of a closed Riemannian
surface is the length of a shortest non-contractible closed curve. The systole $\sys(\Gamma,h)$ of a metric
graph is the length of a shortest cycle. A graph $\Gamma\subset M$, with the metric induced by $h$,
\emph{captures the topology} of $M$ if the map $H_1(\Gamma;\R)\to H_1(M;\R)$ induced by the inclusion is onto.
Let $L(M,h)$ be the infimum of the lengths of graphs that capture the topology of $M$. A graph capturing
the topology of $M$ with length $L(M,h)$ is called \emph{minimal} \cite[Definition~5.5]{Kar15}. For $u\in M$
and $r>0$, let $B(u,r)$ be the closed ball of radius $r$ centred at $u$ in $(M,h)$, and let
$B^+(u,r)\supset B(u,r)$ be the set obtained by filling with $2$-cells the boundary components of $B(u,r)$
that are contractible in $M$ \cite[Definition~6.4]{Kar15}.

\subsection{Part (1)}
We follow the proof of \cite[Theorem~9.2]{Kar15} and only replace the input from graphs. Fix $R>0$.
Since surface groups are residually finite, there is a finite cover $\bar M\to M$ such that $\bar M$ is
orientable and $\sys(\bar M,\bar h)\ge\max\{2R,\frac12\}$, where $\bar h$ is the lift of $h$. Lifting $h$ and $hyp$
multiplies both areas by the degree of the cover, and $\bar M$ has the same universal cover as $M$. So we may
assume that $M$ is orientable of genus $g\ge2$ and that $\sys(M,h)\ge\max\{2R,\frac12\}$.

By the Gauss--Bonnet formula, $\area(M,hyp)=4\pi(g-1)$, so
$\area(M,h)\le\frac{1}{2^{13}\pi}\cdot4\pi(g-1)\le\frac{1}{2^{12}}(2g-1)$. Hence \cite[Proposition~8.1]{Kar15}
applies and provides a minimal graph $\Gamma\subset M$ with $\length(\Gamma,h)\le\frac12(2g-1)$. Since $\Gamma$ captures
the topology of $M$, we have $b(\Gamma)\ge2g$, and therefore
\begin{equation}\label{eq:shortgraph}
\length(\Gamma,h)\ \le\ \tfrac12\,(b(\Gamma)-1)\ =\ \tfrac16\,\length(\Gamma_{b(\Gamma)}).
\end{equation}
The proof of \cite[Theorem~9.2]{Kar15} establishes two further facts for every vertex $u$ of $\Gamma$ and
every $r\in(0,R)$. First, by minimality of $\Gamma$,
\[
\length\partial B^+(u,r)\ \ge\ \length\big(\Gamma\cap B^+(u,r)\big)\ \ge\ \length B_\Gamma(u,r).
\]
Second, $\sys(\Gamma,h)\ge\sys(M,h)>2r$, so the covering map restricts to a length-preserving bijection from
$B_{\Gt}(\tilde u,r)$ onto $B_\Gamma(u,r)$ for every lift $\tilde u$ of $u$.

By \eqref{eq:shortgraph}, Theorem~\ref{thmA} applies with $\lambda=\frac16$, so $3\lambda=\frac12$. It gives a vertex
$\tilde u$ of $\Gt$, the same for all radii, with $\length B_{\Gt}(\tilde u,r)\ge\cV_{1/2}(r)$ for every $r\ge0$.
Let $u$ be the image of $\tilde u$ in $\Gamma$. Since $\partial B^+(u,r)\subset\partial B(u,r)$, the coarea formula and
Lemma~\ref{lem:beta}(3) yield
\[
\begin{aligned}
\area B(u,R)\ &\ge\ \int_0^R\length\partial B^+(u,r)\,dr\ \ge\ \int_0^R\cV_{1/2}(r)\,dr\\
&\ge\ \int_0^R\tfrac32\,(4^r-1)\,dr=\tfrac32\Big(\frac{4^R-1}{\ln4}-R\Big).
\end{aligned}
\]
Since $R\le\frac12\sys(M,h)$, the ball $B(u,R)$ lifts isometrically to a ball of radius $R$ in $\Mt$. This proves the
first inequality of part~(1).

Write $G(R):=\frac32(\frac{4^R-1}{\ln4}-R)$ and $H(R):=V_{\Hyp^2}(R)=2\pi(\cosh R-1)$. Then
$G(3)>63.6>57.0>H(3)$. For $R\ge3$,
\[
G'(R)=\tfrac32(4^R-1)\ \ge\ \pi e^R\ \ge\ 2\pi\sinh R=H'(R).
\]
The middle inequality holds at $R=3$, where $94.5>63.2$, and it persists for $R\ge3$ because
$\frac{d}{dR}\log(4^R-1)>\ln4>1=\frac{d}{dR}\log(\pi e^R)$. Hence $G\ge H$ on $[3,\infty)$.

\subsection{Part (2)}
Let $\delta_0=\frac{1}{2^{13}\pi}$ and suppose $\area(M,h)\le\frac{\delta_0}{16}\area(M,hyp)$. Choose $\mu>0$ with
$\area(M,\mu^2h)=\delta_0\area(M,hyp)$, so $\mu\ge4$. Since $V_{(\Mt,\mu^2\htil)}(\mu R)=\mu^2V_{(\Mt,\htil)}(R)$, part~(1)
applied to $\mu^2h$ gives
\[
V_{(\Mt,\htil)}(R)\ \ge\ \mu^{-2}\,G(\mu R).
\]
For $x\ge3$ we have $G(x)\ge0.97\cdot4^x$. Indeed, $G(x)-0.97\cdot4^x\ge0.112\cdot4^x-1.083-1.5x$, which is
positive at $x=3$ and increasing. For $R\ge1$ and $\mu\ge4$ we have $\mu R\ge4$. The function
$\mu\mapsto\mu R\ln4-2\ln\mu$ is increasing for $\mu\ge4$, so
\[
\mu^{-2}G(\mu R)\ \ge\ \frac{0.97}{16}\,e^{4R\ln4}=\frac{0.97}{16}\,256^R\ \ge\ \pi e^R\ \ge\ H(R),
\]
where the middle inequality reduces to $(256/e)^R\ge16\pi/0.97$, which holds for $R\ge1$ since
$256/e>94>52>16\pi/0.97$. This proves part~(2). \qed

\section{Sharpness and the conjecture}\label{sec:sharp}

For $\lambda>0$, let $\Gamma_b^\lambda$ denote the graph $\Gamma_b$ with all edges of length $\lambda$.

\begin{proposition}\label{prop:optimal}
Conjecture~\ref{conjD}, if true, is optimal in two senses.
\begin{enumerate}
\item For every $\lambda>0$, the graph $\Gamma_b^\lambda$ satisfies the hypothesis of Conjecture~\ref{conjD}, and every
vertex $\tilde u$ of its universal cover satisfies $\length B(\tilde u,R)=\cV_\lambda(R)$ for every $R\ge0$. Hence
$\cV_\lambda$ cannot be replaced by $\cV_{\lambda'}$ with $\lambda'<\lambda$.
\item The threshold $\lambda\le1$, under which the conjecture gives $\length B(\tilde u,R)\ge\cV_1(R)$ for all
$R$, is optimal. For $\lambda\in(1,2)$, every vertex $\tilde u$ of the universal cover of $\Gamma_b^\lambda$ satisfies
$\length B(\tilde u,\lambda)=3\lambda<6\lambda-3=\cV_1(\lambda)$.
\end{enumerate}
\end{proposition}

\begin{proof}
The universal cover of $\Gamma_b^\lambda$ is $T_\lambda$, which gives the equality in (1). If $\lambda'<\lambda$ and
$R\in(\lambda',\lambda]$, then $\cV_{\lambda'}(R)>3R=\cV_\lambda(R)$ by Definition~\ref{def:Vc}. Part (2) follows from the
same formula: $\cV_\lambda(\lambda)=3\lambda$ and $\cV_1(\lambda)=3+6(\lambda-1)$ for $\lambda\in[1,2]$.
\end{proof}

\noindent\textbf{Where the factor $3$ comes from.}
Both ingredients of Theorem~\ref{thmA} are sharp: Theorem~\ref{thmB} for $T=T_c$, and Lemma~\ref{lem:core}
by Remark~\ref{rem:core-sharp}. Their extremal configurations are incompatible, however. Let $\Gamma_0$ be a
subgraph as in Lemma~\ref{lem:core}, and put $b_0=b(\Gamma_0)$. If all branch vertices of $\Gamma_0$ were trivalent and
all its arcs had length close to $c$, then $\Gamma_0$ would have $3b_0-3$ arcs and length close to $3c(b_0-1)$. This is
three times the bound $c(b_0-1)$ guaranteed by Lemma~\ref{lem:core}(1). Conversely, the rose of
Remark~\ref{rem:core-sharp} has a branch vertex of very large degree and therefore very large balls. A proof of
Conjecture~\ref{conjD} along these lines would need a comparison principle for trees whose edges are short only
\emph{on average}, weighted by degrees as in Lim's formula for the entropy-minimising metric \cite{Lim08}.

\begin{example}[Equality for non-equilateral theta graphs]\label{ex:theta}
Let $\Theta$ be the theta graph whose three edges have lengths $x\le y\le z$, where $x+y+z=3=\length(\Gamma_2)$ and
$z\le\frac32$. Then $z\ge1$ and $x+y\ge z$. The universal cover of $\Theta$ is a trivalent tree in which each vertex
meets one edge of each length, and all its vertices are equivalent under isometries. Fix a vertex $\tilde u$
and, for $r>0$, let $N(r)$ be the number of points at distance exactly $r$ from $\tilde u$. Geodesics issuing
from $\tilde u$ follow edge paths that never traverse an edge immediately followed by the same edge in the
opposite direction. Such paths reach a second vertex after length $x$, $y$ or $z$, and a third vertex after
length at least $x+y\ge z$. So $N=3,4,5$ on $(0,x),(x,y),(y,z)$ respectively, and
\[
\length B(\tilde u,z)=\int_0^zN(r)\,dr=3x+4(y-x)+5(z-y)=6z-3=\cV_1(z).
\]
Moreover $\length B(\tilde u,R)\ge\cV_1(R)$ on $[0,z]$. On $[0,1]$ this holds because $N\ge3$ and
$\cV_1(R)=3R$ there. On $[1,z]$ it holds because $\cV_1$ has slope $6$ there while $R\mapsto\length B(\tilde u,R)$
has slope at most $5$, and the two agree at $z$. Thus Conjecture~\ref{conjD} for $\lambda=1$ is attained with
equality at a positive radius by a two-parameter family of non-equilateral graphs. In particular, no
strengthening of the form $\length B\ge(1+\varepsilon)\cV_1$ is possible at all radii.
\end{example}


\end{document}